\documentclass[11pt]{article}

\usepackage{amsmath,amsthm,amssymb, amsfonts,amscd}
\usepackage{mathtools}
\usepackage{bbm}
\usepackage{mathrsfs}
\usepackage{enumitem} 
\usepackage[utf8]{inputenc}
\usepackage[dvipsnames]{xcolor}
\usepackage{centernot}

\renewcommand{\epsilon}{{\varepsilon}}

\makeatother

\newtheorem{theorem}{Theorem}[section]
\newtheorem{lemma}[theorem]{Lemma}

\newtheorem{proposition}[theorem]{Proposition}

\theoremstyle{definition}
\newtheorem{definition}[theorem]{Definition}
\newtheorem{remark}[theorem]{Remark}

\newcommand{\Z}{\mathbb{Z}}
\newcommand{\R}{\mathbb{R}}
\newcommand{\N}{\mathbb{N}}

\newcommand{\E}{\mathbb{E}}
\newcommand{\prob}{\mathbb{P}}

\newcommand{\nlr}[4]{#3\mathrel{\mathop{\centernot\longleftrightarrow}_{#1}^{#2}} #4}

\allowdisplaybreaks[1]

\numberwithin{equation}{section}

\usepackage{titlesec}
\titleformat*{\section}{\large\bfseries}
\titlelabel{\thetitle.\quad}
\titleformat{\subsection}[runin]{\normalfont\bfseries}{\thesubsection.}{.5em}{}[.]\titlespacing{\subsection}{0pt}{2ex plus .1ex minus .2ex}{.8em}

\begin{document}

\title{{Chemical Distance for the Level Sets of the Gaussian Free Field}}

\author{
\normalsize{\textsc{Tal Peretz}\footnote{Technion - Israel Institute of Technology. 
E-mail: tal.peretz@campus.technion.ac.il}\ \ \ }}

\maketitle

\begin{abstract}
We consider the Gaussian free field $\varphi$ on $\Z^d$ for $d \geq 3$ and study the level sets $\{\varphi \geq h \}$ in the percolating regime. We prove upper and lower bounds for the probability that the chemical distance is much larger than Euclidean distance. Our proof uses a renormalization scheme combined with a bootstrap argument.
\newline
MSC 2020 \emph{subject classifications.}  60K35, 82B43.
\end{abstract}

\section{Introduction}
Percolation is one of the central topics in probability theory and over the past two decades there has been active research in such models which have long-range correlation. In this article we study a canonical example, the level-sets of the Gaussian free field (GFF) on $\Z^d$ for $d \geq 3$. This subject was first studied in \cite{BricmontLebowitzMaesPercolationGaussian}, and later reintroduced in \cite{RodriguezSznitmanPhaseTransition}. Since then, there has been much progress in understanding its percolation properties.

Denote $\{\varphi_x : x \in \Z^d\}$ to be the GFF on $\Z^d$, whose distribution we denote by $\prob$. More concretely, this is the centered Gaussian field such that $\E[\varphi_x \varphi_y] = g(x,y)$, where $g$ is the Green's function of the simple random walk on $\Z^d$, see \eqref{eq: green} for the definition. For a fixed height $h \in \R$, we are interested in the set
\begin{align*}
E^{\geq h} = \{x \in \Z^d: \varphi_x \geq h \},
\end{align*}
which we consider as a subgraph of $\Z^d$. Let $\{ 0 \xleftrightarrow{ \geq h} \infty\}$ denote the event there exists an infinite connected subset of $E^{\geq h}$ which contains the origin, and define the critical height
\begin{align*}
h_* =h_*(d)=\inf \{h \in \R: \prob[0 \xleftrightarrow{\varphi \geq h} \infty]= 0 \}.
\end{align*}
Rodriguez and Sznitman in \cite{RodriguezSznitmanPhaseTransition} showed that this parameter is critical in the following sense:
\begin{itemize}
\item[--] for $h<h_*$, $\prob$-a.s. $E^{\geq h}$ contains a unique infinite connected component,
\item[--] for $h>h_*$, $\prob$-a.s. $E^{\geq h}$ consists only of finite connected components.
\end{itemize} 
Today much is known about $h_*$, including that $h_* \in (0,\infty)$, see \cite{DrewitzPrevostRodriguezSign,RodriguezSznitmanPhaseTransition}, and $h_*(d) \sim \sqrt{2 \log d}$ as $d \to \infty$, see \cite{DrewitzRodriguezHigh}. Furthermore, from \cite{EqualityParametersGFF,GoswamiRodriguezSeveroRadius} we have that the level sets are in a strongly supercritical regime when $h<h_*$, and in a strongly subcritical regime when $h>h_*$: for all $h \neq h_* $
\begin{align} \label{eq: sharp-transition}
\prob[0 \xleftrightarrow{ \geq h} \partial B_N, \nlr{}{\geq h}{0}{\infty} ] \leq \begin{cases} e^{-cN/ \log N} & \text{for }d = 3 \\ e^{-cN} & \text{for }d \geq 4 \end{cases},
\end{align}
where the event refers to the origin lying in a finite connected component of $E^{\geq h}$ which intersects the boundary of $B_N = \{x \in \Z^d: \vert  x\vert_\infty \leq N \}$.

This article is interested in the graph distance on the level sets in the supercritical regime $h<h_*$. We define the chemical distance of $x,y \in E^{\geq h}$ as
\begin{align*}
\rho_h(x,y) = \inf\{n \in \N: \exists z_1,\ldots, z_n \in E^{\geq h} \: \text{s.t.} \: \vert z_{i+1}-z_i \vert_1 = 1, z_1 = x, z_n = y \},
\end{align*}
where we use the convention $\inf \varnothing = \infty$. For $h < h_*$, let $ \mathcal S_N(h)$ be the vertices in $E^{\geq h}$ which are in connected components with $\vert \cdot \vert_\infty$-diameter greater than $N$. The chemical distance of $E^{\geq h}$, as well as for many other percolation models, was first studied in \cite{DrewitzRathSapozhnikovChemical}. The authors showed that with very high probability the chemical distance of the GFF is comparable to the Euclidean norm: for $h<h_*$ there exists constants $\Delta = \Delta(d,h)>0$, $C = C(d,h)>0$ and $c = c(d,h)>0$ such that 
\begin{align}\label{eq: initial-chemical-bound}
\prob[\exists x,y \in \mathcal S_N(h) \cap B_N, \: \rho_h(x,y) > C N]\leq \exp \left(-c  (\log N)^{1+\Delta} \right).
\end{align}
Their proof is a multiscale argument which is robust enough to apply to many percolation models with long-range correlation. However, the stretched exponential bound is a byproduct of their methods and is not expected to be sharp for our case.
\begin{theorem} \label{thm: chem-bound}
For $d \geq 3$ and $h< h_*$, there exists $c=c(d,h)>0$ and $C=C(d,h)>0$ such that
\begin{align*}
\prob[\exists x,y \in \mathcal S_N(h) \cap B_N, \: \rho_h(x,y) > C N]\leq \exp \left(-c  N^{1-2/d} \right).
\end{align*}
\end{theorem}
We also provide complementary lower bounds.
\begin{theorem}\label{thm: chem-lower-bound}
For $h< h_*$ and $\alpha>1$, there exists $c=c(d,h,\alpha)>0$ such that
\begin{align*}
\prob[\exists x,y \in \mathcal S_N(h) \cap B_N, \: \rho_h(x,y) > \alpha  N] \geq \begin{cases} e^{-cN/\log N} & \text{ for } \: d = 3 \\ e^{-cN} &\text{ for }\: d \geq 4 \end{cases}.
\end{align*}
\end{theorem}
The upper bound improves on \eqref{eq: initial-chemical-bound}, however there remains a gap between the upper and lower bounds. We give intuition on the exponents in both theorems and for the discrepancy between them. Following the proof of Theorem \ref{thm: chem-bound}, one sees that the upper bound is dominated by the event the chemical distance is larger than $N$ in a box of radius $N^{1/d}$, which has exponential cost $N^{1-2/d}$. On the other hand, the lower bound is derived by forcing the path between two vertices to make a large detour of size $N$, which has exponential cost $N/\log N$ and $N$ for $d = 3$ and $d \geq 4$, respectively. As a point of comparison, for Bernoulli percolation the probability of the chemical distance being larger than Euclidean distance decays exponentially, see \cite{AntalPisztoraChemical}.

The technical contribution of this work is to use a renormalization scheme from \cite{SznitmanDisconnection} in order to bootstrap the estimate \eqref{eq: initial-chemical-bound}. We will partition $B_{N}$ into boxes $U$ of side-length $L $, which we will eventually take to infinity, and check the local connectivity of the level sets inside each box. To decouple the field, we will use the Markov property of the GFF: $\varphi= \psi^U + \xi^U$, where $\psi^U$ is a local independent GFF, and $\xi^U$ is a harmonic average field. We define a box $U$ being good if the local field $\psi^U$ has typical chemical distance, and that the harmonic average $\xi^U$ is not too small inside the box. By construction, the chemical distance $\rho_h$ behaves typically inside a box which is good with respect to $\psi^U$ and $\xi^U$. Using \eqref{eq: initial-chemical-bound} and the independence of the fields $\psi^U$, we can show that most boxes are good with respect to the local field. In order to estimate the number of good boxes with respect to the harmonic average, we will use a Gaussian estimate from \cite{SznitmanDisconnection}, see Lemma \ref{lem: harmonic-control}. On the event the number of bad boxes is not too large, we will be able to bound the chemical distance between any two connected points in $B_N$. 

The article is organized in the following way. In Section \ref{sec: notation}, we introduce notation and preliminary results for the level sets of GFF. In Section \ref{sec: renormalization}, we set up a renormalization scheme. We define notions of good and bad boxes, and prove bounds on the probability of having many bad boxes. In Section \ref{sec: connectivity}, we show that on the event of having few bad boxes, the chemical distance in $E^{\geq h}$ behaves well. In Section \ref{sec: proof} we gather all of the results to prove Theorem \ref{thm: chem-bound}, and in Section \ref{sec: lowerbound-gff} we prove Theorem \ref{thm: chem-lower-bound}.\\

Throughout the rest of this text, we denote $c,c',C,C',\ldots$ to be generic numbers in $(0,\infty)$ which change from line to line, while numbered constants $c_1,c_2,\ldots$ will be fixed throughout the text. We will note if they depend on parameters, with the exception of the dimension $d$. Lastly, some of our inequalities will only hold for large $N$ and $L$.

\section{Preliminaries} \label{sec: notation}
\subsection{Notation}
For a field $\chi: \Z^d \to \R$, we write $\{\chi \geq h \} = \{z \in \Z^d: \chi_z \geq h \} $. For two sets $U,V\subset \Z^d$, we define $\{U \xleftrightarrow{ \chi \geq h} V \}$ to be the event there exists a nearest-neighbor path in $\{\chi \geq h \} $ starting in $U$ and ending in $V$. For shorthand, we sometimes write $\{U \xleftrightarrow{  \geq h} V \}=\{U \xleftrightarrow{ \varphi \geq h} V \}$. For $x \in \Z^d$ we denote $\{x \xleftrightarrow{ \geq h} \infty \}$ to be the event $x$ lies in the unique infinite connected component of $E^{\geq h}$. We also define the complement of these events
\begin{align*}
\{ \nlr{}{\geq h}{U}{V} \} = \{U \xleftrightarrow{\geq h} V \}^c, \quad \{ \nlr{}{\geq h}{x}{\infty} \} = \{x \xleftrightarrow{\geq h} \infty \}^c.
\end{align*}

Denote $\vert \cdot \vert_p$ for $p \in \{1,2,\infty\}$ to be the usual $\ell_p$ norm on $\Z^d$. Define $B_N(x) = \{y \in \Z^d: \vert x - y \vert_\infty \leq N \}$ and $B_N = B_N(0)$. For finite $U \subset \Z^d$, we define its boundary $\partial U = \{ x \in U^c : \exists y \in U, \vert x - y  \vert_1 = 1 \}$. \\

For positive sequences $a_n$ and $b_n$, we write  $a_n \ll b_n$ if $\lim_{n \to \infty} a_n/b_n = 0$, $a_n = O(b_n)$ if $\limsup_{n \to \infty} a_n /b_n < \infty$ and $a_n \asymp b_n $ if $0 < \liminf_{n \to \infty} a_n/b_n \leq \limsup_{n \to \infty} a_n/b_n <\infty.$\\

Let $\{X_ n : n \in \N\}$ denote the simple random walk on $\Z^d$, and let $P^x$ denote its law conditioned on starting at $x \in \Z^d $. For $d \geq 3$, let 
\begin{align} \label{eq: green}
g(x,y) = \sum_{n=0}^\infty P^{x}[X_n = y]\quad \text{for } x,y \in \Z^d
\end{align}
denote its Green's function. We recall the well known fact
\begin{align} \label{eq: green-asymptotic}
g(x,y) \asymp \vert x -y \vert_\infty^{2-d} \:\: \text{as} \:\: \vert x - y \vert_\infty \to \infty.
\end{align}
For $U \subset \Z^d$, let $T_U = \inf \{ n \geq 0 , X_n \not \in U \}$ denote the exit time of $U$, and let $H_U = \inf \{n \geq 1 : X_n \in U \}$ be the hitting time of $U$. Define the Green's function killed outside of $U$
\begin{align*}
g_U(x,y) = \sum_{n=0}^\infty P^x(X_n = y , n < T_U)\quad \text{for } x,y \in \Z^d.
\end{align*}
For $U \subset \subset \Z^d$, define the equilibrium measure of $U$
\begin{align*}
e_U(x) = P^x[ H_U = \infty] \quad\text{for } x \in \Z^d,
\end{align*}
and the capacity of $U$ 
\begin{align*}
\text{Cap}(U)  = \sum_{x \in U }e_U(x).
\end{align*}
For $K \subset \subset U$, define the equilibrium measure of $K$ relative to $U$
\begin{align*}
e_{K,U}(x) = P^x[ H_K > T_U] \quad\text{for } x \in \Z^d,
\end{align*}
and the capacity of $K$ relative to $U$
\begin{align*}
\text{Cap}_U(K)  = \sum_{x \in K }e_{K,U}(x).
\end{align*}
We observe that $\text{Cap}_{\Z^d}(K)  = \text{Cap}(K)$.
\subsection{Covariance Structure}
A central element of our proof is the Gibbs-Markov decomposition which is expressed in the following lemma, see for instance \cite{SznitmanDisconnection}.
\begin{proposition} \label{prop: Gibbs-Markov}
For finite $U \subset \Z^d$, we have
\begin{align*} 
\varphi_x = \psi_x^U + \xi_x^U
\end{align*}
where 
\begin{itemize}
\item $\psi^U$ is a centered Gaussian field with covariance $g_U(\cdot,\cdot)$, independent of $\sigma(\varphi_z: z \in U^c)$.
\item  $\xi^U$ is the unique harmonic function on $U$ with boundary condition $\xi \vert_{U^c} = \varphi \vert_{U^c}$:
\begin{align*}
\xi^U_x = E^x[\varphi_{X_{T_U}}] = \sum_{y \in \Z^d} P^x[X_{T_U}=y] \varphi_y.
\end{align*}
\end{itemize} 
In particular, $\psi^U$ and $\xi^U$ are independent of each other.
\end{proposition}
We introduce integers $L \geq 1$ and $K\geq 100$. We will eventually let $L$ go to infinity and let $K$ be some large fixed constant. Define the lattice 
\begin{align*}
\mathbb L = L\Z^d.
\end{align*} 
For $z \in \mathbb L$, define the boxes
\begin{align*}
C_z = z+ [0,L)^d \subset D_z = z+[-3L,4L)^d \subset U_z = z+[-KL+1,L+KL-1)^d.
\end{align*}
Given a subset $\mathcal C \subset \mathbb L$, the next two lemmas will be used to decouple the fields $\varphi^{U_z},\xi^{U_z}$ for $  z \in \mathcal C$.
\begin{lemma}[{\cite[Lemma 4.1]{SznitmanDisconnection}}] \label{lem: indep-control}
Let $\mathcal C \subset \mathbb L$ be a collection of sites with mutual $\vert \cdot \vert_\infty$-distance at least $(2K+1)L$. Then the fields $\psi^{U_z}$, $z \in \mathcal C$, are independent.
\end{lemma}

\begin{lemma}[{\cite[Corollary 4.4]{SznitmanDisconnection}}] \label{lem: harmonic-control}
For all $K \geq 100$ and $a>0$, there exist constants $c = c(K)$ and $c'= c'(K)$ such that
\begin{equation*}
\limsup_{L}\sup_{\mathcal C} \left \{ \log \prob \left[ \bigcap_{z \in \mathcal C}\left\{\inf_{ y \in D_z} \xi_y^{U_z} \leq -a\ \right\} \right] + c'\left( a -c \cdot \sqrt{\frac{\vert \mathcal C \vert}{\operatorname{Cap}(\Sigma)}} \right)_+^2\operatorname{Cap}(\Sigma)\right \} \leq 0
\end{equation*}
where $\Sigma = \bigcup_{z \in \mathcal C} C_z$, and the supremum runs over all $\mathcal C \subset \mathbb L$ with mutual $\vert \cdot \vert_\infty$-distance at least $(2K+1)L$.  We also have for every $z \in \Z^d$
\begin{align*}
\prob\left[\sup_{y \in D_z} \vert \xi^{U_z}_y \vert \geq a \right] \leq 2 \exp \left( - c'L^{d-2}\left(a - \frac{c}{L^{(d-2)/2}} \right)_+^2\right).
\end{align*}
\end{lemma}

\subsection{Connectivity Estimates}
We recall connective properties of the percolation of the level sets of $\varphi$. 
\begin{definition}
Denote $\chi$ to be either $\psi^{U_z}$ or $\varphi$. For $L\geq 1$ and $h_1,h_2 \in \R$ such that $h_1\leq h_2$, define $\text{LocUniq}(\chi, z,h_1,h_2)$ as the intersection of the events
\begin{equation*}
\text{Exist}(\chi,z,h_2) = \left\{ \begin{alignedat}{2} \begin{gathered} \text{there exists a connected component in} \\[.25cm]  \{\chi \geq h_2 \}\cap C_z \: \text{with diameter at least} \: L/10 \end{gathered} \end{alignedat} \right \}
\end{equation*}
and 
\begin{equation*}
\text{Unique}(\chi, z,h_1,h_2) = \left\{ \begin{alignedat}{2} \begin{gathered} \text{for any } x \in \mathbb L \text{ with } \vert z - x \vert_1 = L\text{ , any connected clusters in}  \\[.25cm]  \{\chi \geq h_2 \}\cap C_x \: \text{and } \{\chi \geq h_2 \}\cap C_z \: \text{having diameter at least} \: L/10  \\[.25cm] \text{are connected to each other in} \:    \{\chi \geq h_1 \}\cap D_z \end{gathered} \end{alignedat} \right \}.
\end{equation*}
\end{definition}
\begin{lemma}[\cite{DrewitzRathSapozhnikovChemical},\cite{EqualityParametersGFF}] \label{lem: local-unique}
 For $h_1 \leq h_2< h_{*}$, there exists constants $c_1 = c_1(h_1,h_2)$ and $c=c(h_1,h_2)$ such that for any $z \in \Z^d$
\begin{align*}
\prob[\operatorname{LocUniq}(\varphi,z,h_1,h_2)] > 1-  e^{-c L^{c_1}} .
\end{align*}
\end{lemma}
The next result, which was already mentioned in the introduction, will be our initial estimate for controlling the chemical distance.
\begin{lemma}[\cite{DrewitzRathSapozhnikovChemical},\cite{EqualityParametersGFF}]\label{lem: initial-chemical-bound}
For $h_1\leq h_2<h_*$, there exists $C_1 = C_1(h_1,h_2)$, $c=c(h_1,h_2)$, and $c_2 = c_2(h_1,h_2)$ such that for any $z \in \Z^d$
\begin{align*}
 \prob[\forall x,y \in \mathcal S_{L/10}(h_2) \cap D_z, \: \rho_{h_1}(x,y) \leq C_1 L] >1- e^{-c (\log L)^{1+c_2}}.
\end{align*}
\end{lemma}
\section{Renormalization} \label{sec: renormalization}
In this section, we setup a renormalization argument, which we will use to prove the upper bound in Theorem \ref{thm: chem-bound}.
\subsection{Good and Bad Boxes}
We define notion of good and bad vertices in $\mathbb L$.
\begin{definition}
For $\varepsilon>0$, we say $z \in \mathbb L$ is $\xi$-good at level $\varepsilon$ if
\begin{align*}
\inf_{x \in D_z}\xi_x^{U_z} > -\varepsilon.
\end{align*}
Else, we say $z$ is $\xi$-bad at level $\varepsilon$.
\end{definition}
Denote $\mathcal S^\psi_{z}  (h) $ to be connected components in $\{x \in \Z^d: \psi^{U_z}_x \geq h \}$ with diameter at least $L/10$. We define the chemical distance with respect to the field $\psi^{U_z}$: for $x,y \in \{w \in \Z^d: \psi^{U_z}_w \geq h \}$,
\begin{align*}
 \eta_{z,h}(x,y) = \inf\{n \in \N: \exists z_1,\ldots, z_n \in \mathcal S^\psi_{z}(h)\: \text{s.t.} \: \vert z_{i+1}-z_i \vert_1 = 1, z_1 = x, z_n = y \}.
\end{align*}

\begin{definition}
For $h_1\leq h_2$, we say $z \in \mathbb L$ is $\psi$-good at level $(h_1,h_2)$ if the event
\begin{align*}
\text{LocUniq}(\psi^{U_z},z,h_1,h_2) \cap \{\forall x,y \in \mathcal S^\psi_{z}(h_2) \cap D_z, \: \eta_{z,h_1}(x,y) \leq C_1 L \}
\end{align*}
occurs. Else, we say $z$ is $\psi$-bad at level $(h_1,h_2)$.
\end{definition}
Finally, we will define a notion of a vertex being bad with respect to both the independent field and the harmonic part. 
\begin{definition}
We say $z \in \mathbb L$ is good at level $(\epsilon, h_1,h_2)$ if it is both $\xi$-good at level $\varepsilon$ and $\psi$-good at level $(h_1,h_2)$. Else, we say it is bad at level $(\epsilon, h_1,h_2)$.
\end{definition}
\subsection{Bounding the number of bad boxes}
The main result of this section is the following proposition, which estimates the number of bad vertices in a large box. We let $N \in \mathbb N$, which we will eventually take to infinity, and fix $h_1<h_2 < h_*$ and $\epsilon >0$. Denote $\mathcal B_N =\mathcal B_N(\epsilon,h_1,h_2)= \{x \in \mathbb L \cap B_N: x \text{ is } (\epsilon,h_1,h_2) \text{-bad} \}$.

\begin{proposition} \label{prop: bad-box}
There exist constants $c = c(K)$, $c'= c'(K)$ and $C = C(K)$ such that for $L < N$ and $m \in \N$ satisfying $m^{2/d}/L^{d-2} \leq c'\varepsilon^2 $, we have
\begin{align*}
\prob[\vert \mathcal{B}_N(\epsilon,h_1,h_2) \vert \geq m] \leq \exp(C m \log N-c m (\log L)^{1+c_2})+ \exp(C m \log N-c \varepsilon^2 m^{1-2/d}L^{d-2}).
\end{align*}
\end{proposition}

To prove this proposition, we will first estimate the probability of any fixed configuration of vertices in $\mathbb L$ are either $\psi$-bad or $\xi$-bad. 

\begin{lemma}\label{lem: indep-bad}
Suppose $\mathcal C \subset \mathbb L $ is a set of points with mutual distance at least $(2K+1)L$. There exists $c=c(\varepsilon,h)$ independent of $\mathcal C$ such that
\begin{align*}
\prob[\forall z \in \mathcal C, \text{ $z$ is $\psi$-bad at level $(h_1,h_2)$}] \leq  \exp \left (-c\cdot  \vert \mathcal C \vert  (\log L)^{1+c_2} \right ).
\end{align*}
\end{lemma}
\begin{proof}
By Lemma \ref{lem: indep-control} and the definition of being $\psi$-bad, the events $\{z \text{ is } \psi\text{-bad at level }(h_1,h_2) \}$ for $z \in \mathcal C$ are independent. By the union bound,
\begin{equation}
\begin{aligned} \label{eq: indep-bad-box}
\prob[z \text{ is } \psi\text{-bad at level }(h_1,h_2)] & \leq \prob[\text{LocUniq}(\psi^{U_z},z,h_1,h_2)^c]\\
&+\prob[ \exists x,y \in \mathcal S^\psi_{z}(h_2) \cap D_z, \: \eta_{z,h_1}(x,y) > C_1 L ],
\end{aligned}
\end{equation}
and so we need to bound both terms on the right hand side. By translation invariance of $\prob$, we only need to consider the case where $z$ is the origin. Let
\begin{align*}
\delta = \frac{(h_2-h_1) \wedge (h_*-h_2)}{4},
\end{align*}
which satisfies 
\begin{align}\label{eq: height-satisfaction}
h_1+\delta<h_2 - \delta,\qquad  h_2 + \delta < h_*.
\end{align}
We make the following observation: on the event 
\begin{align*}
A_\delta = \{ \sup_{x \in D_0} \vert \varphi_x-\psi_x^{U_0} \vert \leq \delta \},
\end{align*}
 we have for any $h' \in \R$
\begin{align} \label{eq: sprinkle-inclusion}
\{x \in D_0 : \varphi_x \geq h'+\delta \} \subset \{x \in D_0: \psi^{U_0}_x \geq h' \} \subset \{x \in D_0: \varphi_x\geq h'-\delta \} 
\end{align}
We bound the first term on the right hand side of \eqref{eq: indep-bad-box}. From the union bound and Lemma \ref{lem: harmonic-control}, we have
\begin{equation}\label{eq: loc-unique-loc-field}
\begin{aligned}
\prob[\text{LocUniq}(\psi^{U_0},0,h_1,h_2)^c] &\leq \prob[\text{Exist}(\psi^{U_0},0,h_1,h_2)^c, A_\delta ]\\
&+\prob[\text{Unique}(\psi^{U_0},0,h_1,h_2)^c,A_\delta] \\
&+ \exp\left( -c\delta^2L^{d-2} \right).
\end{aligned}
\end{equation}
Since $\text{Exist}(\chi,0,h_2)$ is decreasing in $h_2$, by \eqref{eq: sprinkle-inclusion} we have
\begin{align*}
\prob[\text{Exist}(\psi^{U_0},0,h_2)^c,A_\delta] \leq \prob[\text{Exist}(\varphi,0,h_2+\delta)^c] \leq e^{-cL^{c_1}},
\end{align*}
where the last inequality follows from Lemma \ref{lem: local-unique} and \eqref{eq: height-satisfaction}. 

To bound the second term on the right hand side of \eqref{eq: loc-unique-loc-field}, observe that on $A_\delta$, the first inclusion in \eqref{eq: sprinkle-inclusion} implies that for $x,y \in \Z^d$
\begin{align}\label{eq: local-field-inclusion}
\{x \xleftrightarrow{\varphi \geq h_1 + \delta} y \} \subset \{x \xleftrightarrow{ \psi^{U_0}\geq h_1} y \},
\end{align} 
while the second inclusion implies $\mathcal S_0^\psi(h_2) \subset \mathcal S_{L/10}(h_2-\delta) $. 
Hence on the event $A_\delta$
\begin{align*}
\text{Unique}(\varphi,0,h_1 + \delta,h_2-\delta)  \subset \text{Unique}(\psi^{U_0},0,h_1,h_2),
\end{align*}
and so by Lemma \ref{lem: local-unique} and \eqref{eq: height-satisfaction} we have
\begin{align*}
\prob[\text{Unique}(\psi^{U_0},0,h_1,h_2)^c] \leq  e^{-cL^{c_1}}.
\end{align*}
Thus $\prob[\text{LocUniq}(\psi^{U_0},0,h_1,h_2)^c] \leq e^{-cL^{c_1}}$.

We now bound the second term on the right hand side of \eqref{eq: indep-bad-box}. On the event $A_\delta$ we have $\eta_{0,h_1}(x,y) \leq \rho_{h_1+\delta}(x,y)$ by \eqref{eq: local-field-inclusion}, and also $\mathcal S_0^\psi(h_2) \subset \mathcal S_{L/10}(h_2-\delta) $. Hence on $A_\delta$ we have
\begin{align*}
\{\exists x,y \in \mathcal S_0^\psi(h_2) \cap D_0,\eta_{0,h_1}(x,y) > C_1L \}  \subset \{\exists x,y \in \mathcal S_{L/10}(h_2-\delta)\cap D_x,\rho_{h_1+\delta}(x,y) > C_1L \}
\end{align*}
and so 
\begin{align*}
&\prob[ \exists x,y \in \mathcal S^\psi_{0}(h_2) \cap D_x, \: \eta_{x,h_1}(x,y) > C_1 L  ] \\
&\leq \prob[\sup_{y \in D_0} \vert \varphi_y-\psi_y^{U_0} \vert> \delta] +  \prob[\exists x,y \in \mathcal S_{L/10}(h_2-\delta) \cap C_x, \: \rho_{h_1+\delta}(x,y) > C_1 L].
\end{align*}
By Lemmas  \ref{lem: harmonic-control} and \ref{lem: initial-chemical-bound}, and \eqref{eq: height-satisfaction} we have
\begin{align*}
\prob[\exists x,y \in \mathcal S_{L/10}(h_2-\delta) \cap C_x, \: \rho_{h_1+\delta}(x,y) > C_1 L]  \leq\exp(-cL^{d-2})+ \exp(-c (\log L)^{1+c_2}) ,
\end{align*}
which finishes the proof.
\end{proof}

Before we prove the equivalent result for the harmonic components, we will need the following lower bound on the capacity of separated boxes. 

\begin{lemma} \label{lem: capacity}
Let $m \in \N$, and suppose $\{ z_1,\ldots, z_m\} \subset\mathbb L $ is a subset of points at mutual distance at least $(2K+1)L$. Then there exists $c_3 = c_3(K)$ independent of $L$ and $m$ such that
\begin{align*}
\operatorname{Cap}\left( \cup_{i=1}^m C_{z_i}\right) \geq c_3\cdot m^{1-2/d}L^{d-2}.
\end{align*}
\end{lemma}
\begin{proof}
Let $U = \cup_{i=1}^m C_{z_i}$. From \cite[(2.6)]{GoswamiRodriguezSeveroRadius}, we have
\begin{align*}
\text{Cap}(U) \geq \frac{\vert U \vert}{\max_{x \in U} \sum_{y \in U} g(x,y)},
\end{align*}
 and so we need to bound $\max_{x \in U} \sum_{y \in U} g(x,y)$. By \eqref{eq: green-asymptotic}, without loss of generality we can assume $U \subset B_{\lfloor Cm^{1/d}L \rfloor}(z)$ for some $z \in \Z^d$ and large $C$.
 
For any $x \in U$, we have the bound
\begin{align*}
\sum_{y \in U} g(x,y) = \sum_{r=0}^\infty \sum_{y \in \mathcal D_r(x)} g(x,y)
\end{align*}
where $\mathcal D_r(x) = U \cap \{y \in \Z^d:  \overline{L}r \leq \vert x-y \vert_\infty < \overline{L}(r+1) \}$ for $\overline{L} = (1+2K)L$. Since $\vert z_i - z_j \vert_\infty \geq \overline{L}$ for $i \neq j$, we have
\begin{align*}
\{ z_1,\ldots, z_m \} \cap  \{x \in \Z^d:  \overline{L}r \leq \vert z-x \vert_\infty \leq \overline{L}(r+1) \}\leq C r^{d-1}.
\end{align*} 
In particular, we have $\vert \mathcal D_r \cap U \vert \leq Cr^{d-1} L^d$. Since $\vert x - y \vert_\infty \geq \overline L r$ for $y \in \mathcal D_r$, by \eqref{eq: green-asymptotic} we have $g(x,y) \leq C (rL)^{2-d}$. Hence
\begin{align*}
\sum_{y \in U} g(x,y) \leq \sum_{r=0}^{\lfloor C'm^{1/d} \rfloor } C (rL)^{2-d}r^{d-1}L^d \leq C m^{2/d} L^{2},
\end{align*}
and we conclude that
\begin{align*}
\text{Cap}(U) \geq \frac{cmL^{d}}{m^{2/d}L^{2}} = c\cdot m^{1-2/d}L^{d-2}.
\end{align*}
\end{proof}
\begin{lemma}\label{lem: harmonic-bad}
Suppose $\mathcal C$ is a subset of $\mathbb L $ of points at mutual distance at least $(2K+1)L$. There exist $c = c(K)$ and $c'=c'(K)$ such that for every $\epsilon>0$ and $m,L \in \N$ satisfying $m^{2/d}/L^{d-2} \leq c'\epsilon^2$, and any $\{z_1,\ldots, z_m \} \subset \mathcal C$,
\begin{align*}
\prob[\text{$z_i$ is $\xi$-bad at level $\epsilon$ for }i=1,\ldots, m] \leq \exp\left(-c \epsilon^2 \operatorname{Cap}(\cup_{i=1}^m C_{z_i}) \right).
\end{align*}
\end{lemma}
\begin{proof}
We apply Lemma \ref{lem: harmonic-control} 
\begin{align*}
&\prob[\text{$z_i$ is $\xi$-bad at level $\epsilon$ for }i=1,\ldots, m]\\
&= \prob \left[ \inf_{ x \in D_{z_i}} \xi_x^{U_{z_i}} \leq -\epsilon, \text{for }i = 1,\ldots, m \right] \\
&\leq \exp\left( -c'\left( \epsilon -c\sqrt{\frac{m}{\text{Cap}(\cup_{i=1}^m C_{z_i})}} \right)_+^2\text{Cap}(\cup_{i=1}^m C_{z_i})\right).
\end{align*}
Combining Lemma \ref{lem: capacity} with the assumption $m^{2/d} \leq c' \epsilon^2 L^{d-2}$, we have
\begin{align*}
\frac{m}{\text{Cap}(\cup_{i=1}^m C_{z_i})} \leq \frac{m}{c_3m^{1-2/d}L^{d-2}} \leq \frac{c' \epsilon^2}{c_3}.
\end{align*}
Hence for small enough $c'$, we have
\begin{align*}
\prob[\text{$z_i$ is $\xi$-bad at level $\epsilon$ for }i=1,\ldots, m]\leq  \exp\left( - c\epsilon^2 \text{Cap}(\cup_{i=1}^m C_{z_i}) \right).
\end{align*}
This finishes the proof.
\end{proof}
With Lemmas \ref{lem: indep-bad} and \ref{lem: harmonic-bad}, we will prove in a straightforward manner Proposition \ref{prop: bad-box}.
\begin{proof}[Proof of Proposition \ref{prop: bad-box}]
Define the events
\begin{equation*}
\mathcal E = \{ \exists z_1,\ldots, z_{\lfloor m/2 \rfloor} \in \mathbb L \cap B_N: z_1, \ldots, z_{\lfloor m/2 \rfloor} \: \text{are}  \: \psi\text{-bad at level }(h_1,h_2)\}
\end{equation*}
and 
\begin{equation*}
\mathcal F = \{ \exists z_1,\ldots, z_{\lfloor m/2 \rfloor} \in \mathbb L \cap B_N: z_1, \ldots, z_{\lfloor m/2 \rfloor} \: \text{are}  \: \xi\text{-bad at level }\epsilon\},
\end{equation*}
so that by the union bound
\begin{align*}
\prob[\vert \mathcal{B}_N(\epsilon,h_1,h_2) \vert \geq m] \leq \prob[\mathcal E] + \prob[\mathcal F].
\end{align*}
Define $\mathcal A_{N,L} = \vert \{z \in \mathbb L : C_z \cap B_N \neq \varnothing \}\vert$, and note that $\mathcal A_{N,L}  \leq C N^{d}/L^d$. We bound the first term on the right hand side. We observe that if there are $\lfloor m/2 \rfloor $ boxes, we can choose in some fixed deterministic way a subset of $ m' = C m/(2K+1)^d$ boxes which are $(2K+1)L$ separated. We thus have
\begin{align*}
\prob[\mathcal E] \leq \binom{\mathcal A_{N,L}}{ m' }  \sup_{\{z_1,\ldots, z_{m'} \}}\prob[z_1, \ldots, z_{m'} \: \text{are}  \: \psi\text{-bad at level }(h_1,h_2)],
\end{align*}
where the supremum is over all $ \{ z_1,\ldots, z_{m'} \} \subset \mathbb L \cap B_N$ with mutual distance at least $(2K+1)L$. Applying Lemma \ref{lem: indep-bad} and the bounds $\binom{n}{k} \leq n^k$ and $ m' \leq C m/(2K+1)^d$ yields us
\begin{align*}
\prob[\mathcal E]  \leq \exp(C' m (\log m + \log N) - c m (\log L)^{1+c_2})
\end{align*}
for $C' = C/(2K+1)^d$.
We now bound the second term. The previous separation argument, combined with Lemmas \ref{lem: capacity} and \ref{lem: harmonic-bad}, implies
\begin{align*}
\prob[\mathcal E] \leq \binom{\mathcal A_{N,L}}{ m' } \exp(-c \epsilon^2 m^{1-2/d} L^{d-2}) \leq \exp(C' m\log N  -c \epsilon^2 m^{1-2/d} L^{d-2}).
\end{align*}
This finishes the proof.
\end{proof}
\section{Connectivity} \label{sec: connectivity}
In this section, we will construct a deterministic path between any two points in the same connected component based on arguments from \cite[Section 3]{AntalPisztoraChemical}. First, we show that we can construct a path in $E^{\geq h}$ along a sequence of good boxes.

\begin{proposition} \label{prop: path}
 If $\{z_1,\ldots, z_n\} \subset \mathbb L$ is a sequence of nearest-neighbor points in $\mathbb L$ which are all good at level $(\epsilon,h_1,h_2)$, then there exists a path in 
\begin{align*}
E^{\geq h_1 -\epsilon} \cap \left(\cup_{i=1}^n D_{z_i} \right)
\end{align*}
starting at $C_{z_1}$ and ending in $C_{z_n}$ whose length is bounded by $C_1 L\cdot n$.
\end{proposition}
\begin{proof}
Suppose $z_j$ and $z_{j+1}$ are nearest-neighbor points in $\mathbb L$. Since they are both $\psi$-good at level $(h_1,h_2)$, this implies both $C_{z_j}$ and $C_{z_{j+1}}$ contain connected components of $\{\psi^{U_{z_j}} \geq h_2  \}$ and $\{\psi^{U_{z_{j+1}}} \geq h_2 \}$, respectively, with diameter greater than $L/10$. Furthermore, they are connected in $\{ \psi^{U_z} \geq h_1\} \cap D_{z_j}$ and the chemical distance $ \eta_{z_j,h_1}$ between any two points in any of these connected components in $D_{z_j}$ is bounded by $C_1L$. Since $z_j$ and $z_{j+1}$ are both $\xi$-good at level $\epsilon$, these connectivity properties extend to $(D_{z_j} \cup D_{z_{j+1}})\cap E^{\geq h_1-\epsilon}$. Using induction finishes the proof.
\end{proof}

We introduce more notation. We say $\gamma \subset \mathbb L$ is a $*$-connected path if $\gamma =  (z_1,\ldots, z_n)$ and $\vert z_{j+1} - z_j \vert_\infty = L$ for $j=1,\ldots,n-1$. We say a set $U \subset \mathbb L$ is $*$-connected if for any $x,y \in U$, there exists a $*$-connected path between $x$ and $y$ contained in $U$. Denote $\mathscr{C} $ to be the collection of  $*$-connected components of $\{z \in \mathbb L : z \text{ is } (\epsilon,h_1,h_2) \text{ bad}  \}$. For $z \in \mathbb L$, denote $\mathbf{C}_z$ to be the element of $\mathscr{C} $ containing $z$. If $z$ is good, denote $\mathbf{C}_z = \varnothing$ and $\partial \mathbf{C}_z= \{ z\}$. For $x \in \Z^d$, let $\ell(x) \in \mathbb L$ be the unique point such that $x \in C_{\ell(z)}$.  The following result follows from \cite[Proposition 3.1]{AntalPisztoraChemical}.
\begin{proposition}[{\cite[Proposition 3.1]{AntalPisztoraChemical}}] \label{prop: amend-path} 
Fix $x,y \in \Z^d$ and a $*$-connected path $\sigma = ( \sigma_1, \ldots, \sigma_n) \subset \mathbb L$ with $\sigma_1 = \ell(x)$ and $\sigma_n = \ell(y)$. On the event $\{x \xleftrightarrow{ \geq h_1-\epsilon} y \}$, there exists a self-avoiding path $\gamma \subset E^{\geq h_1 - \epsilon}$ connecting $x$ and $y$ such that
\begin{align*}
\gamma \subset W = \bigcup_{j=1}^n \bigcup_{z \in \overline{ \mathbf C}_{\sigma_j}} D_{ z}
\end{align*}
where $\overline{ \mathbf C}_{\sigma_j} = \mathbf C_{\sigma_j} \cup \partial\mathbf C_{\sigma_j}  $
\end{proposition}
\begin{remark}
This proposition follows from the proof of \cite[Proposition 3.1]{AntalPisztoraChemical}. While their setting is Bernoulli bond percolation, their proof is a deterministic construction of amending a nearest-neighbor path inside a bond percolation cluster such that it lies in the boundary and interior of a cluster of bad boxes. Their proof does not rely on the distribution of the cluster, rather on the macroscopic properties of good boxes, which they denote as a `white boxes'. Their definition of a good box, see \cite[(2.9)]{AntalPisztoraChemical}, is not the same as ours. However, the only property the authors use of good boxes is \cite[(2.13)]{AntalPisztoraChemical}, which we can replace with Proposition \ref{prop: path}.
\end{remark}
\section{Proof of Theorem \ref{thm: chem-bound}} \label{sec: proof}
Fix $h<h_*$, $\epsilon \in (0,(h_*-h)/4)$ and $N \in \mathbb N$. Let $h_1 = h+\epsilon<h_*$ and $h_2 = h+2\epsilon<h_*$. Let $M>C/\epsilon^2$ for some large $C$, and let 
\begin{align*}
m_N = \left\lfloor \frac{N^{1-2/d}}{\log N}  \right \rfloor, \qquad L_N = \lfloor M (N/m_N)^{1/d} \rfloor= \lfloor M( N^{2/d} \log N)^{1/d} \rfloor.
\end{align*}
Note that by our choice $m_N L_N^d \asymp N$.
\begin{lemma} \label{lem: bound-chem}
Fix $x,y \in B_N$. On the event
\begin{align*}
\{x \xleftrightarrow{\geq h} y \} \cap \{ \vert \mathcal B_{2N}(\epsilon,h_1,h_2) \vert \leq m_N \}
\end{align*}
 there exists $C_2 = C_2(h,\epsilon) $ such that
\begin{align*}
\rho_h(x,y) \leq C_2N.
\end{align*}
\end{lemma}
\begin{proof}
Let $x,y \in B_N$ in the same cluster, and let $\sigma= (\sigma_1,\ldots, \sigma_n)$ be a $*$-connected path of vertices in $\mathbb L$ such that $\sigma_1 = \ell(x)$ and $\sigma_n = \ell(y)$. Since $x,y \in B_N$, we can assume $n \leq C N/L_N$. By Proposition \ref{prop: amend-path}, there exists a path $\gamma \subset E^{\geq h_1-\epsilon}=E^{\geq h}$ connecting $x$ and $y$ such that $\gamma \subset W$, where
\begin{align*}
\gamma \subset W = \bigcup_{j=1}^n \bigcup_{z \in \overline{ \mathbf C}_{\sigma_j}} D_{ z}.
\end{align*}
Since $x,y \in B_N$ and $B_{2N}$ has at most $m_N$ $(\epsilon,h_1,h_2)$-bad boxes, we infer that the diameter of $\overline{ \mathbf C}_{\sigma_j}$ is at most $7u_N L_N$ for $j=1,\ldots, n$. Since $7u_NL_N \ll N$, we conclude that $W \subset B_{2N}$. Since $B_{2N}$ contains $W$ and has at most $m_N$ bad boxes, this implies that $W$ intersects at most $(2^d+1) m_N$ boxes which are either $(\epsilon,h_1,h_2)$-bad or $*$-neighbors of one. Hence the path crosses at most $CN/L_N + (2^d+1) m_N$ good boxes, and at most $m_N$ bad boxes. By Proposition \ref{prop: path}, the chemical distance inside a good box is bounded by $C_1 L_N$, while for a bad box a trivial upper bound for the chemical distance is $C'L_N^d$ for some $C'$. We thus get
\begin{align*}
\vert \gamma \vert \leq  \sum_{j=1}^n \sum_{z \in \overline{ \mathbf C}_{\sigma_j}} \max_{x,y \in D_z}\rho_h(x,y) \leq (C N/L_N+ (2^d+1) m_N) \times C_1 L + m_N \times C'L^d \leq C_2 N,
\end{align*}
 for some $C_2=C_2(h,\epsilon)$ since $m_NL_N^d \asymp N$.
\end{proof}

\begin{figure}[t!]  \centering
    \includegraphics[scale=0.50]{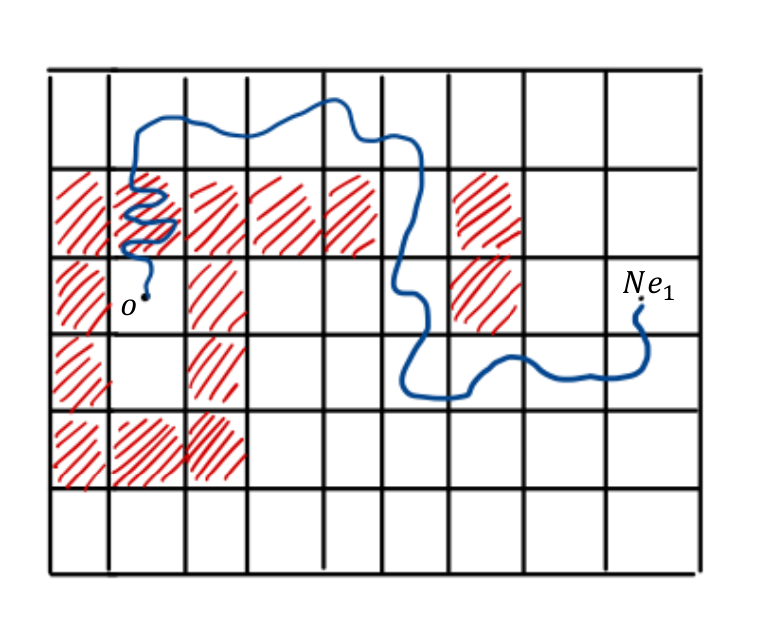}
    \caption{\textit{The red boxes represent bad boxes, while the blue line represents a connected subset of the level sets that connects the origin to $Ne_1$. Crossing a bad box, which is unavoidable in this example, incurs a cost of $O(L^d)$, whereas avoiding bad boxes costs $O(L)$. }}
\end{figure}

\begin{proof}[Proof of upper bound in Theorem \ref{thm: chem-bound}]
We first decompose our probability
\begin{align*}
\prob[\exists x,y \in \mathcal S_N(h) \cap B_N, \rho_h(x,y) > C N] &\leq \prob[\exists x,y \in \mathcal S_N(h) \cap B_N,   \nlr{}{\geq h}{x}{y} ]  \\
&+ \prob[\exists x,y \in \mathcal S_N(h) \cap B_N,  x \xleftrightarrow{\geq h} y, \rho_h(x,y) > C N].
\end{align*}
To bound the first term, we make a few observations. First, if $\{\nlr{}{\geq h}{x}{y}\}$ occurs, then necessarily either $x$ or $y$ does not lie in the infinite connected component. Second, if $x$ is in a connected component of diameter at least $N/10$, then it is connected to the boundary of $B(x,N/10)$. Hence we have
\begin{align*}
\{ \exists x,y \in \mathcal S_N(h) \cap B_N,   \nlr{}{\geq h}{x}{y} \} &\subset \{ \exists x \in \mathcal S_N(h) \cap B_N,   \nlr{}{\geq h}{x}{\infty}\}\\ 
&\subset \{ \exists x \in B_N  \text{ s.t. } x \xleftrightarrow{\geq h}\partial B(x,N/10),   \nlr{}{\geq h}{x}{\infty}\}.
\end{align*}
 From \eqref{eq: sharp-transition}, a union bound, and translation invariance, we have for $d \geq 3$
\begin{align*}
\prob[\exists x \in B_N  \text{ s.t. } x \xleftrightarrow{\geq h}\partial B(x,N/10),   \nlr{}{\geq h}{x}{\infty}] \leq \exp(-c N/\log N)
\end{align*}
which implies
\begin{align*}
\prob[\exists x,y \in \mathcal S_N(h) \cap B_N,   \nlr{}{\geq h}{x}{y} ] \leq e^{-c N/\log N}. 
\end{align*}
We bound the second term. By Lemma \ref{lem: bound-chem}, we get
\begin{align*}
\prob[\exists x,y \in \mathcal S_N(h) \cap B_N,  x \xleftrightarrow{\geq h} y,  \rho_h(x,y) > C_2 N] \leq \prob[\vert \mathcal B_{2N}(\epsilon,h_1,h_2 )\vert \geq m_N].
\end{align*}
By our choice of $m_N$ and $L_N$ we have
\begin{align*}
\frac{m_N^{2/d}}{L_N^{d-2}} \leq \frac{m_N^{2/d}}{M \log N \cdot m_N^{2/d} } = o(1) ,
\end{align*}
and so we can apply Proposition \ref{prop: bad-box}. We then get 
\begin{align*}
\prob[\vert \mathcal B_{2N}(\epsilon,h )\vert \geq m_N]& \leq \exp(C m_N \log N-c m_N (\log L_N)^{1+c_2})+ \exp(C m_N \log N-c \epsilon^2 m_N^{1-2/d}L_N^{d-2}) \\
&\leq \exp(-c N^{1-2/d} (\log N)^{c_2})+ \exp(-c \epsilon^2 N^{1-2/d}) \\
&\leq 2 \exp(-c \epsilon^2 N^{1-2/d}) 
\end{align*}
for large enough $N$, where the second inequality follows from $\epsilon^2 m_N^{1-2/d}L_N^{d-2} \geq C m_N \log N$ by our choice of $m_N$ and $L_N$. We thus have
\begin{align*}
\prob[\exists x,y \in \mathcal S_N(h) \cap B_N, \rho_h(x,y) > C_2 N] \leq \exp(-c N/\log N) + 2 \exp(-c \epsilon^2 N^{1-2/d}) 
\end{align*}
which finishes the proof.
\end{proof}
\section{Proof of Theorem \ref{thm: chem-lower-bound}}\label{sec: lowerbound-gff}
In this section we will prove the lower bounds using techniques from \cite{GoswamiRodriguezSeveroRadius}. We will need the following general result for lower bounds for the GFF. For $U \subset \Z^d$, denote $\prob_U$ to be the law of $\psi^U$. Given $A \in \mathcal B(\R^K)$, $K \subset \Z^d$, and $h \in \R$, we define
\begin{align}\label{eq: shift-event}
A^h = \{\varphi\vert_K - h \in A \}
\end{align} 
where $\varphi\vert_K - h $ refers to the field restricted to $K$ shifted by $-h$ coordinate-wise. 
\begin{lemma}[{\cite[Lemma 3.2]{GoswamiRodriguezSeveroRadius}}]\label{lem: cap-lowerbound}
Let $U_N \subset \subset V_N \subset \Z^d$ be subsets with $\operatorname{Cap}_{V_N}(U_N) \to \infty$. Let $A_N \in \mathcal B(\R^{U_N})$ and $I \subset \R$ be an interval such that, for every $h' \in I$,
\begin{align*}
\prob_{V_N}[A_N^{h'}] \to 1.
\end{align*}
Then for every $h \not \in I$,
\begin{align*}
\liminf_{N \to \infty} \frac{1}{\operatorname{Cap}_{V_N}(U_N)} \log \prob_{V_N}[A_N^h] \geq -\frac{1}{2}d(h,I)^2.
\end{align*}
\end{lemma}
Our strategy to prove the lower bound will be to create a long path between the two points inside $E^{\geq h}$ which is insulated by $\{ \varphi < h \}$. To decouple the increasing event the points are connected, and the decreasing event the path is insulated, we will use the Gibbs-Markov decomposition. 

\begin{proof}[Proof of Theorem \ref{thm: chem-lower-bound}]
We first consider the more involved case $d=3$. For $r \geq 0$ and $n \in \N$, define the $r$-neighborhood of the line segment connecting $0$ to $ne_1$
\begin{align*}
P^{(1)}_{n,r} = [-r,n+r] \times  [-r,r]^{d-1} \cap \Z^d
\end{align*}
and the $r$-neighborhood of the line segment connecting $0$ to $ne_2$
\begin{align*}
P^{(2)}_{n,r} = [-r,r] \times [-r,n+r] \times [-r,r]^{d-2}\cap \Z^d.
\end{align*}
We fix $\alpha,N \in \N$ and define the set
\begin{align*}
P_{r} = P_r(N,\alpha) = P^{(2)}_{\alpha N,r} \cup (\alpha N e_2 + P^{(1)}_{N,r}) \cup (Ne_1+P^{(2)}_{\alpha N,r}).
\end{align*}
We fix $\epsilon>0$ and define the sets $U_N \subset V_N \subset W_N$ by $U_N=P_{\lfloor N^{\epsilon} \rfloor}$, $V_N=P_{ \lfloor N^{2\epsilon} \rfloor}$ and $W_N=P_{\lfloor N^{3\epsilon} \rfloor}$.

For a field $\chi: \Z^d \to \R$, define the events
\begin{align*}
D_N(\chi,h) = \{ B_{\lfloor N^{\epsilon/2} \rfloor}(0) \: \text{and} \: B_{\lfloor N^{\epsilon/2} \rfloor}(Ne_1) \text{ are connected by a path in } \{\chi \geq h \} \cap U_N\}
\end{align*}
and
\begin{align*}
F_N(h) =  \{ \nlr{}{\geq h}{\partial V_N }{W_N} \}.
\end{align*}
Since
\begin{align*}
D_N(h) \cap F_N(h)  \subset \bigcup_{x \in \partial B_{\lfloor N^{\epsilon/2} \rfloor}(0)}\bigcup_{y \in \partial B_{\lfloor N^{\epsilon/2} \rfloor}(Ne_1)} \{x \xleftrightarrow{\geq h} y, \: \rho_h(x,y)\geq (2\alpha+1) N-2\lfloor N^{\epsilon/2} \rfloor \},
\end{align*}
we have
\begin{align*}
D_N(h) \cap F_N(h)  \subset \{\exists x,y \in \mathcal S_N(h) \cap B_N, \: \rho_h(x,y) > \alpha N \}
\end{align*}
for large enough $N$.
\begin{figure}[t!]  \centering
    \includegraphics[scale=0.40]{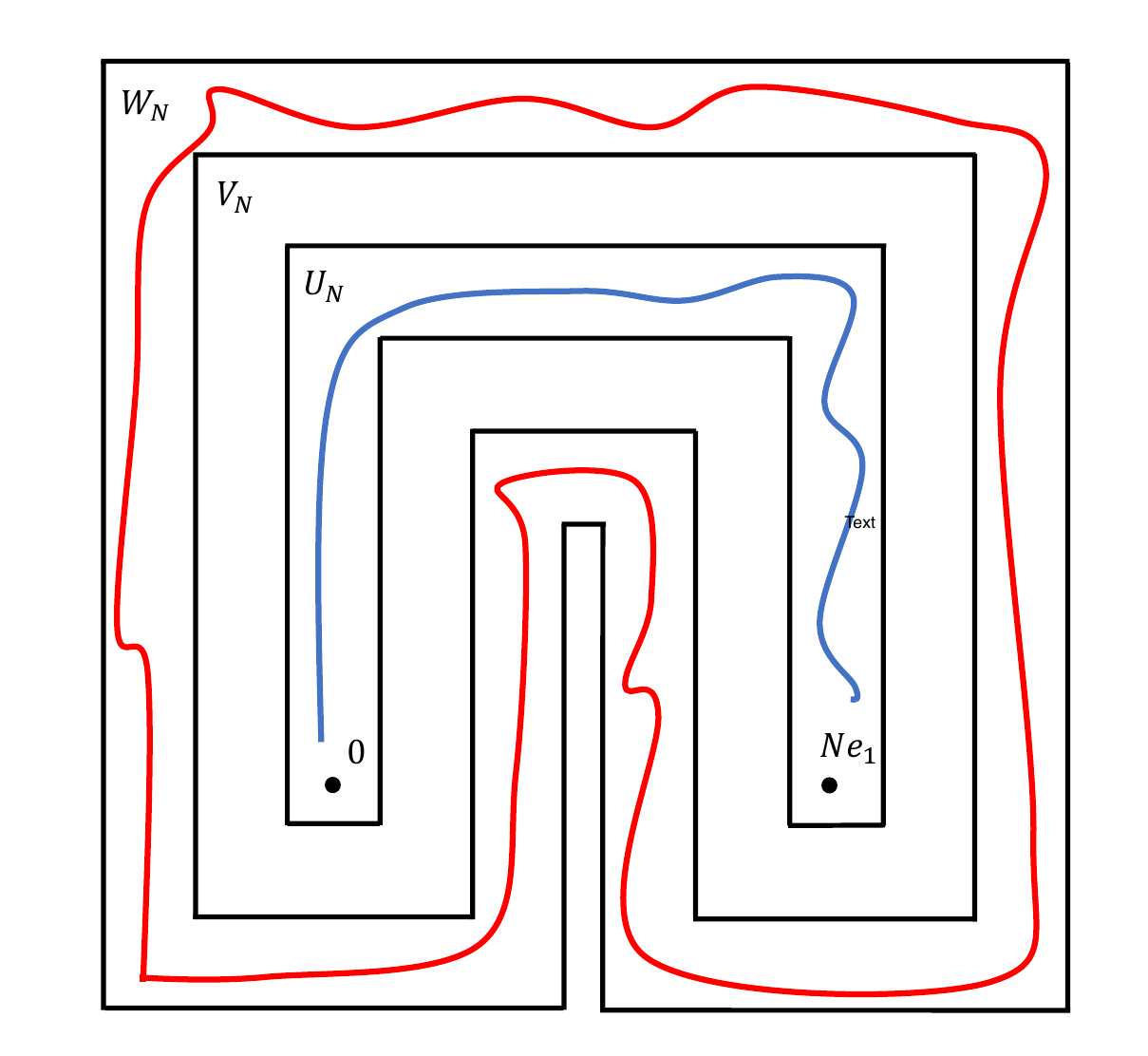}
    \caption{\textit{The event $D_N(h)$ creates a long path (in blue) inside $U_N \cap E^{\geq h}$, while the event $F_N(h)$ insulates (in red) the path, forcing the chemical distance to be at least $\alpha N$.}}
\end{figure}

To derive a lower bound for $\prob[D_N(h) \cap F_N(h) ]$, fix $\delta>0$ and define the event
\begin{align*}
E_N(h,\delta) = \left \{ \inf_{y \in U_N} \xi_y^{V_N} \geq -(h_* - h+\delta) \right \}.
\end{align*}
Since $D_N$ is an increasing event, and by Proposition \ref{prop: Gibbs-Markov}, we have
\begin{align*}
\prob[D_N(\varphi, h) \cap E_N(h,\delta) \cap F_N(h)] &\geq \prob[D_N(\psi^{V_N}, h_*+\delta) \cap E_N(h) \cap F_N(h)] \\
&= \prob[D_N(\psi^{V_N}, h_*+\delta)]\prob[ E_N(h,\delta) \cap F_N(h)].
\end{align*}
We will now derive lower bounds for both terms. We first claim that for $h \in (h_*,h_*+ \delta)$, $\prob[ E_N(h,\delta) \cap F_N(h)] \to 1$ as $N \to \infty$. From \eqref{eq: sharp-transition}, we have for $R \in \N$ and $h>h_*$
\begin{align*}
\prob[0 \xleftrightarrow{ \geq h} \partial B_R ] \leq Ce^{-c R/\log R} .
\end{align*}
Applying a union bound and translation invariance with this estimate for $R = \lfloor N^{3 \epsilon} \rfloor - \lfloor N^{2 \epsilon} \rfloor $, we have $\prob[F_N(h)] \to 1$ for $h>h_*$. To bound $\prob[E_N(h,\delta)]$, we first bound the variance of $\xi_x^{V_N} $ for $x \in U_N$. Following the computation from \cite{GoswamiRodriguezSeveroRadius}, see the equation below (3.15), we have
\begin{align*}
\E[(\xi^{V_N}_{x})^2] \leq C \cdot \text{dist}(x, V_N^c)^{-(d-2)}\leq CN^{-2\epsilon(d-2)}.
\end{align*}
Using Gaussian tail estimates and a union bound implies 
\begin{align*}
\prob \left[ \inf_{y \in U_N} \xi_y^{V_N} \geq -(h_* - h+\delta) \right] \to 1
\end{align*}
 for $h\in(h_*, h_* + \delta)$. We conclude that $\prob[ E_N(h,\delta) \cap F_N(h)] \to 1$ for $h\in(h_*, h_* + \delta)$ . Note that the events $ E_N(h,\delta)$ and $F_N(h)$ take the form in the assumption for Lemma \ref{lem: cap-lowerbound}. For $E_N(h,\delta)$, this is because we can write
\begin{align*}
\left \{ \inf_{y \in U_N} \xi_y^{V_N} \geq -(h_* - h+\delta) \right \} = \left \{ \inf_{y \in U_N} \sum_{x \in V_N}P^y[X_{T_{V_N}}=x](\varphi_x -h)\geq -(h_* +\delta) \right \} .
\end{align*}
Hence we can apply Lemma \ref{lem: cap-lowerbound} with $V_N = \Z^d$ and $U_N = W_N$ and conclude that for  $h<h_*$ and $\delta>0$
\begin{align*}
\prob[ E_N(h,\delta) \cap F_N(h)] \geq \exp(-c\operatorname{Cap}(W_N)).
\end{align*}
To derive an upper bound for $\operatorname{Cap}(W_N)$, recall that for $A,B \subset \Z^d$
\begin{align*}
\text{Cap}(A \cup B) \leq \text{Cap}(A) + \text{Cap}(B),
\end{align*}
see \cite[Proposition 2.2.1]{LawlerIntersections}. From \cite[Lemmas 2.2, 2.5]{GoswamiRodriguezSeveroRadius}, we have that for any $\epsilon>0$ and $i \in \{1,2 \}$,
\begin{align*}
\operatorname{Cap}(P^{(i)}_{N,\lfloor N^{\epsilon}\rfloor})\leq CN/\log N .
\end{align*}
We conclude that
\begin{align*}
\text{Cap}(W_N) \leq 2\text{Cap}(P^{(1)}_{ \alpha N,\lfloor N^{3 \epsilon}\rfloor}) + \text{Cap}(P^{(2)}_{\alpha N ,\lfloor N^{3 \epsilon}\rfloor}) \leq C N \log N,
\end{align*}
which implies
\begin{align*}
\prob[ E_N(h) \cap F_N(h)] \geq  e^{-cN/\log N} .
\end{align*}
Next we bound $\prob[D_N(\psi^{V_N}, h_*+\delta)]$. We first claim that for $h<h_*$,  $\prob[D_N(\varphi, h)] \to 1$ as $N \to \infty$. From \cite{SznitmanDisconnection}, we have for $h<h_*$ and $R \in \N$
\begin{align*}
\prob[ \nlr{}{\geq h}{B_R }{\partial B_{2R}} ]  \leq Ce^{-cR^{d-2}}.
\end{align*}
Applying a union bound and translation invariance with this estimate for $R = \lfloor N^{\epsilon /2}\rfloor$ proves the claim. We then have for $h<h_*$ and $ \overline{\delta} < (h_*-h)/2$
\begin{align*}
\prob[D_N(\psi^{V_N}, h)] \geq \prob[D_N(\varphi, h+\overline{\delta}), \sup_{x \in U_N}\xi^{V_N} \leq \overline{\delta}] \geq \prob[D_N(\varphi, h+\overline{\delta})] - \prob[\sup_{x \in U_N}\xi^{V_N} > \overline{\delta}].
\end{align*}
Since $h+\overline{\delta}<h_*$, we have $\prob[D_N(\varphi, h+\overline{\delta})] \to 1$ by the earlier claim. By an earlier calculation, we also have $\prob[\sup_{x \in U_N}\xi^{V_N} > \overline{\delta}] \to 0$. We conclude that $\prob[D_N(\psi^{V_N}, h)] \to 1$ for $h<h_*$. We can now apply Lemma \ref{lem: cap-lowerbound}, and conclude that
\begin{align*}
\prob[D_N(\psi^{V_N}, h_* + \delta)]  \geq \exp(-c\operatorname{Cap}_{V_N}(U_N)).
\end{align*}
We are left to derive an upper bound for $\operatorname{Cap}_{V_N}(U_N)$. Since $ H_{A \cup B} \leq  H_{A}$, we have $P^x[ H_{A \cup B}> T_{U}] \leq P^x[ H_{A }> T_{U}] $. This implies that for $A,B \subset U$
\begin{align*}
\operatorname{Cap}_U(A \cup B) = \sum_{x \in A \cup B}P^x[H_{A \cup B}> T_U] &\leq  \sum_{x \in A}P^x[ H_{A \cup B}> T_U] +  \sum_{x \in B}P^x[ H_{A \cup B}> T_U] \\
&\leq  \sum_{x \in A}P^x[ H_{A }> T_U] +  \sum_{x \in B}P^x[  H_{ B}> T_U] \\
&=  \operatorname{Cap}_U(A) + \operatorname{Cap}_U(B). 
\end{align*}
Hence we have
\begin{align*}
\operatorname{Cap}_{V_N}(U_N)  \leq \text{Cap}_{V_N}( P^{(2)}_{\alpha N,\lfloor N^{\epsilon} \rfloor}) + \text{Cap}_{V_N}(\alpha N e_2 + P^{(1)}_{N,\lfloor N^{\epsilon} \rfloor})+\text{Cap}_{V_N}(Ne_1+P^{(2)}_{\alpha N,\lfloor N^{\epsilon} \rfloor}).
\end{align*}
Note that for $K \subset U_1 \subset U_2$, $\text{Cap}_{U_2}(K) \subset \text{Cap}_{U_1}(K)$, which implies
\begin{align*}
\text{Cap}_{V_N}( P^{(2)}_{\alpha N,\lfloor N^{\epsilon} \rfloor}) \leq \text{Cap}_{P^{(2)}_{\alpha N,\lfloor N^{2\epsilon} \rfloor}}(P^{(2)}_{\alpha N,\lfloor N^{\epsilon} \rfloor}) .
\end{align*}
Applying a similar bound to the remaining terms, and using \cite[Lemmas 2.2 and 2.5]{GoswamiRodriguezSeveroRadius} we conclude that 
\begin{align*}
\operatorname{Cap}_{V_N}(U_N) \leq C N/\log N .
\end{align*}
To summarize, for $h <h_*$
\begin{align*}
\prob[\exists x,y \in \mathcal S_N(h) \cap B_N, \: \rho_h(x,y) > \alpha N ] \geq \prob[D_N(h) \cap E_N(h) \cap F_N(h)] \geq e^{-cN/\log N},
\end{align*}
which finishes the proof for $d = 3$.

For $d \geq 4$, define the set $Z_n = P_0$, where we use the notation from the beginning of the proof. We define the event
\begin{align*}
G_{N}(h) = \{\forall x \in P_0, \: \varphi_x \geq h, \: \forall y \in \partial P_0, \: \varphi_y<h \}
\end{align*}
and observe that
\begin{align*}
G_{N}(h)  \subset \{0 \xleftrightarrow{\geq h} Ne_1, \: \rho_h(x,y) > \alpha N \} \subset \{\exists x,y \in \mathcal S_N(h) \cap B_N, \: \rho_h(x,y) > \alpha N \}.
\end{align*}
Rerunning the FKG-inequality argument in \cite[\S 3.3]{GoswamiRodriguezSeveroRadius} yields
\begin{align*}
\prob[G_{N}(h) ] \geq e^{-cN}
\end{align*}
for $d \geq 4$. This finishes the proof of the theorem.
\end{proof}
\noindent{\bf Acknowledgements.}
I would like to thank Ron Rosenthal for advising me throughout this project and Pierre-Fran\c{c}ois Rodriguez for an insightful discussion.

\bibliographystyle{acm}
\bibliography{references}
\end{document}